\documentclass[11pt,a4paper]{amsart}
\usepackage[mathscr]{eucal}
\usepackage{mathrsfs}
\usepackage{amssymb}
\usepackage{amsxtra}
\usepackage{graphics,color}
\usepackage{pifont}
\usepackage{todonotes}

\newcommand{\pA}{\mathcal{A}}

\newcommand{\pL}{\mathcal{L}}

\newcommand{\pR}{\mathcal{R}}
\newcommand{\pS}{\mathcal{S}}

\newcommand{\pV}{\mathcal{V}}
\newcommand{\pW}{\mathcal{W}}
\newcommand{\sL}{\mathsf{L}}
\newcommand{\sR}{\mathsf{R}}

\newcommand{\eH}{\mathscr{H}}

\newcommand{\eX}{\mathscr{X}}
\newcommand{\eW}{\mathscr{W}}
\newcommand{\eY}{\mathscr{Y}}
\newcommand{\eZ}{\mathscr{Z}}
\newcommand{\eU}{\mathscr{U}}
\newcommand{\eV}{\mathscr{V}}
\newcommand{\bC}{\mathbb{C}}

\newcommand{\bM}{\mathbb{M}}
\newcommand{\bN}{\mathbb{N}}

\newcommand{\Mu}{{\rm M}}

\newcommand{\Tau}{{\rm T}}
\newcommand{\Chi}{{\rm X}}
\newcommand{\Iota}{{\rm I}}
\newcommand{\rL}{{\rm L}}
\newcommand{\rR}{{\rm R}}
\newcommand{\Ele}{{{\mathscr E}\!\mathit{\ell} }}
\newcommand{\Dim}{{\rm dim}}
\newcommand{\Ker}{{\rm ker}}
\newcommand{\Image}{{\rm im}}
\newcommand{\spa}{{\rm span}}
\newcommand{\rk}{{\rm rk}}
\newcommand{\Lann}{{\rm Lann}}
\newcommand{\Rann}{{\rm Rann}}
\newtheorem{theorem}{Theorem}
\newtheorem{proposition}[theorem]{Proposition}
\newtheorem{lemma}[theorem]{Lemma}
\newtheorem{corollary}[theorem]{Corollary}
\theoremstyle{definition}
\newtheorem{example}[theorem]{Example}

\newtheorem{remark}[theorem]{Remark}

\numberwithin{equation}{section}
\begin{document}
%
\title[On the invertibility of  elementary operators]{On the invertibility of  elementary operators}
\author[N. Boudi]{Nadia Boudi}
\address{D\' epartement de Math\' ematiques, Universit\' e Moulay Ismail, Facult\'e des Sciences, Mekn\`es, Maroc}
\email{nadia\_boudi@hotmail.com}
\author[J. Bra\v{c}i\v{c}]{Janko Bra\v{c}i\v{c}}
\address{University of Ljubljana, IMFM, Jadranska ul. 19, SI-1000 Ljubljana, Slovenia}
\email{janko.bracic@fmf.uni-lj.si}
\begin{abstract}
Let $\eX$ be a complex Banach space and $\pL(\eX)$ be the algebra of all bounded linear operators on $\eX$. 
For a given elementary operator $\Phi$ of length $2$ on $\pL(\eX)$, we determine necessary and sufficient conditions for 
the existence of a solution of the equation $\Chi \Phi=0$ in the algebra of all elementary operators on $\pL(\eX)$. 
Our approach allows us to characterize some invertible elementary operators of length $2$ whose inverses are elementary operators.
\end{abstract}
\maketitle

\section{Introduction} \label{sec01}
\setcounter{theorem}{0}

Let $\eX$ be a complex Banach  space, let $\pL(\eX)$ be the algebra of all bounded linear operators on $\eX$, and
let $I$ be the identity operator. For $A,B \in \pL(\eX)$, let $\Mu_{A,B}:~T \mapsto ATB$ be the multiplication operator on  $\pL(\eX)$ induced by $A$ 
and $B$. In particular, $\rL_A=\Mu_{A,I}$ and $\rR_B=\Mu_{I,B}$ are left, respectively right, multiplication operators.

An elementary operator  $\Delta$ on $\pL(\eX)$ is a finite sum of multiplication operators. The length $\ell (\Delta)$ of $\Delta$
is the minimum number of multiplication operators needed in the representation of $\Delta$. Let $\Ele(\pL(\eX))$ stand for  the set of all elementary
operators on $\pL(\eX)$. It is obvious that $\Ele(\pL(\eX))$ is  a subalgebra of $\pL\bigl( \pL(\eX)\bigr)$. 

Our general  purpose  is to find an approach that enables us to characterize  the relationship between non-invertibility
(respectively, invertibility) of an elementary operator and the properties of the defining coefficients. It is clear that an operator
$\Phi \in \pL\bigl( \pL(\eX)\bigr)$ is non-invertible if it is a right zero divisor or a left zero divisor. Our first question is, which properties of the
defining coefficients of a non-invertible elementary operator $\Phi$ of length $2$ make it possible or impossible for $\Phi$ to have a left or a right 
zero divisor in $\Ele(\pL(\eX))$? Roughly speaking, when is the non-invertibility of a length $2$ elementary operator caused by an elementary operator? 
In particular, we show that if a length $2$ elementary operator $\Phi$ is annihilated by an elementary operator, then either there exists a multiplication 
operator $\Mu$ such that $\Mu \Phi=0$ (or $\Phi \Mu=0$), or for every elementary operator $\Psi$ having the same defining spaces, $\Psi$ is non-invertible 
and its non-invertibility is caused by an elementary operator. In the latter case, the defining spaces of $\Phi$ are equivalent to some maximal linear 
spaces of constant rank matrices.  Our study  is based on the description of two-dimensional complex matrix spaces of bounded rank. An interesting 
characterization is that of Atkinson and Stephens \cite{AS}. However, we deal with a different form (see Lemma \ref{brank}). Indeed, we believe that the 
study of equations of the form $\Chi \Phi=0$ for general elementary operators may shed new light on some aspects of matrix spaces with bounded rank.

An important example of length $2$ elementary operators are given by $\Upsilon_{A,B} = \Iota+ \Mu_{A,B}$. We show that
if $\Upsilon_{A,B}$ is invertible and $\Upsilon_{A,B}^{-1}$ is an elementary operator of length $n$, then either $A$ or $B$ is algebraic and
$\min \{\deg(A),\deg(B)\}= n$.
As we shall see, in the case of matrix algebras $\bM_n$, every length $2 $ invertible elementary operator is
a sum of two invertible multiplication operators. Hence all invertible elementary operators are of the form $\Mu_{C,D} \Upsilon_{A,B}$, where $A,B, C, D \in \bM_n$ and
$C, D$ are invertible. However, in the infinite-dimensional case, there are examples of length 2 invertible 
elementary operators, such that all elementary operators having the same defining spaces are invertible and the inverse is of length $2$. 

Our  proofs are elementary in the sense that we use only basic facts from linear algebra and operator theory and basic properties of tensor products.  

The paper is organized as follows. In next section,  various known preliminary results are assembled. In Section \ref{sec03}, we characterize 
elementary operators of length at most $2$ which are annihilated by elementary operators. In Section \ref{sec04} we mainly deal with invertible elementary 
operators of length $2$ whose inverses are elementary operators. In particular, we obtain a complete characterization of invertible length 2 elementary 
operators with inverse of length 2.

\section{Preliminaries} \label{sec02}
\setcounter{theorem}{0}

Let $\eX$ and $ \eY$ be complex Banach spaces. We denote by $\pL(\eX,\eY)$ the space of all bounded linear operators from $\eX$ to $\eY$.
A linear subspace $\pV\subseteq \pL(\eX,\eY)$ is said to be of bounded rank $r$ if $\rk(T) \leq r$, for all $T \in \pV$, 
and it is said to be of constant rank $r$ if $\rk(T)=r$, for all nonzero $T \in \pV$. We will denote by $ \rk(\pV)$ the maximum of the ranks of elements in $\pV$. 
Suppose that $\eX$ and $\eY$ are finite dimensional and that $B$ and $B'$ are bases of $\eX$ and $\eY$, respectively. For $T \in \pV$,  let $M(T, B,B')$ 
denote the matrix representation of $T$ with respect to $B$ and $B'$. Then  $\pV$ is said to be equivalent to the space of matrices $\{ M(T,B,B'):~ T \in \pV\}$.

Denote by $\eX \otimes \eY$ the algebraic tensor product of $\eX$ and $\eY$.
For $x= \sum_{i=1}^n u_i \otimes v_i \in \eX \otimes \eY$, we set
\begin{equation}
\sL(x)= \spa \{u_1, \ldots, u_n\}\qquad \text{and}\qquad \sR(x)= \spa \{v_1, \ldots, v_n\}.
\end{equation}
Let us recall that the rank of $x$ is the minimum number $r(x)$ of simple tensors needed in the representation of $x$.
If $V$ is a vector subspace of $\eX \otimes \eY$, then an element $x \in V$ a minimal tensor of $V$ if, for every nonzero
$y \in V$ such that $r(x)=r(y)+r(x-y)$, one has $y =x$.

Recall the following simple lemma.

\begin{lemma}\label{dim}
Let $x= \sum_{i=1}^n u_i \otimes v_i$. If $x=0$, then $\Dim(\sL(x))+ \Dim(\sR(x))\leq n$.
\end{lemma}

\begin{proof}
Let $\{u_{i_1}, \ldots, u_{i_r}\}$ be a maximal linearly independent subset of $\sL(x)$. With no loss of generality we may assume that $i_t=t$,
for $1 \leq t \leq r$.  Write
$ u_j= \sum_{i=1}^r  \alpha_{ij} u_i$,  for $r+1 \leq j \leq n$.
Then $x= \sum_{i=1}^r u_i \otimes (v_i+ \sum_{j=r+1}^n  \alpha_{ij}v_j)$.
Since $x=0$ and $\{u_1, \ldots, u_r\}$ is a linearly independent set one has $v_1, \ldots, v_r\in \spa \{v_{r+1}, \ldots, v_n\}$.
\end{proof}

For an algebra $\pA$, let $\pA^{op}$ denote the opposite algebra, i.e., the algebra which has the same underlying vector space as $\pA$ but the 
multiplication is given by $x*y=yx$.
The main tool which we use in our study is Theorem 5 and its Corollary  in \cite{Mar}
(see also \cite{FS}).

\begin{lemma}\label{inj} 
The mapping $\varphi: \pL(\eX) \otimes \pL(\eX)^{op} \rightarrow \Ele(\pL(\eX))$, which is defined by 
$\varphi (\sum_{i=1}^n A_i \otimes B_i)= \sum_{i=1}^n \Mu_{A_i, B_i}$,
is an injective homomorphism.
\end{lemma}

Let $\Delta=\sum_{i=1}^n \Mu_{A_i, B_i}$ be an elementary operator. It follows from Lemma \ref{inj} that $\Delta$ has length $n$ if and only if the
corresponding tensor $\sum_{i=1}^n A_i \otimes B_i$ has rank $n$. Here again we set $\sL(\Delta)= \spa \{A_1, \ldots, A_n\}$ and 
$\sR(\Delta)= \spa \{B_1, \ldots, B_n\}$.

For later convenience we state the following corollary of Lemma \ref{inj} (see also \cite[Lemma 1.2]{LLWW}).

\begin{corollary} \label{cor01}
Let $\{ A_1, \ldots, A_m\}$ and $\{ D_1, \ldots, D_n\}$ be linearly independent subsets in $\pL(\eX)$ and let $B_1, \ldots, B_m$,
$C_1, \ldots, C_n$ $\in \pL(\eX)$ be such that
\begin{equation} \label{eq57}
\sum_{i=1}^m \Mu_{A_i, B_i}  + \sum_{i=1}^n \Mu_{C_i, D_i}=0.
\end{equation}
Then $C_i\in \spa\{A_1, \ldots, A_m\}$, for $1\leq i \leq n$, and $B_j \in \spa\{D_1, \ldots, D_n\}$, for $1\leq j \leq m$.
\end{corollary}

\begin{proof}
By Lemma \ref{inj}, $ \sum_{i=1}^m A_i \otimes B_i  + \sum_{i=1}^n C_i \otimes D_i=0$.
Now Lemma \ref{dim} implies that
\begin{equation*}
\Dim (\spa \{A_1, \ldots, A_m, C_1, \cdots C_n \}) + \Dim (\spa \{B_1, \ldots, B_m, D_1, \cdots D_n \}) \leq n+m.
\end{equation*}
However, $\{ A_1, \ldots, A_m\}$ and $\{ D_1, \ldots, D_n\}$ are sets of linearly independent operators. Hence 
$\Dim(\spa \{A_1, \ldots, A_m, C_1, \cdots C_n \}) = m$ and $\Dim(\spa \{B_1, \ldots, B_m, D_1, \cdots D_n \}) = n$. 
We conclude that $C_i\in \spa\{A_1, \ldots, A_m\}$ $(1\leq i \leq n)$ and $B_j \in \spa\{D_1, \ldots, D_n\}$ $(1\leq j \leq m)$.
\end{proof}

We will also need the following simple lemma.

\begin{lemma} \label{lem02}
Let $\Delta$ be an elementary operator of length $n\leq m$. Assume that $\{C_1, \ldots, C_m\}$ is a spanning set of $\mathsf{L}(\Delta)$. 
Then there exists a spanning set  $\{D_1, \ldots, D_m\}$ of $\mathsf{R}(\Delta)$ such that
$\Delta= \sum_{j=1}^{m} \Mu_{C_j,D_j}$.
\end{lemma}

\begin{proof}
Write $\Delta= \sum\limits_{i=1}^{n} \Mu_{A_i, B_i}$. Since $\spa \{C_1, \ldots, C_m\}= \spa \{A_1, \ldots, A_n \}$ there exist numbers 
$\alpha_{ij}$ $(1\leq i\leq m,~$ $1\leq j \leq n)$ such that $ A_j=\sum\limits_{i=1}^m\alpha_{ij}C_i$, for all $j=1, \ldots, n.$
Hence
\begin{equation*}
\sum\limits_{j=1}^n A_j \otimes B_j= \sum\limits_{j=1}^{n}(\sum\limits_{i=1}^m\alpha_{ij}C_i )\otimes B_j=\sum\limits_{i=1}^{m}C_i \otimes (\sum\limits_{j=1}^n\alpha_{ij}B_j).
\end{equation*}
Let $D_i=\sum\limits_{j=1}^n \alpha_{ij}B_j$. Then one has $ \sum\limits_{i=1}^n A_i \otimes B_i=\sum\limits_{i=1}^{m}C_i \otimes D_i$. By Lemma \ref{inj},  $\Delta=\sum\limits_{i=1}^{m}\Mu_{C_i, D_i}$.
\end{proof}

\section{Non-invertibility} \label{sec03}
\setcounter{theorem}{0}

Recall that a nonzero element $a$ in a ring $\pR$ is a left zero divisor if there exists a nonzero $b\in \pR$ such that $ab=0$.
The notion of a right zero divisor is defined similarly. As usual, if $\pS \subseteq \pR$,  then $\Lann(\pS)=\{a \in \pR:  a\pS =0\}$ is
the left annihilator of $\pS$. The right annihilator $\Rann(\pS)$ is defined similarly. Note that an operator $T\in \pL(\eX)$ is a left zero divisor if and only if
$ \Ker T \ne \{ 0\}$ and it is a right zero divisor if and only if $\overline{\Image \;T}\ne \eX$. \par
Let $\Delta \in \Ele(\pL(\eX))$ be an elementary operator of length at most $2$. In this section, the main question is, under which condition the equation
\begin{equation}\label{eq58}
\Chi \Delta=0
\end{equation}
has a solution in $\Ele(\pL(\eX))$. The equation $\Delta \Chi=0$ can be treated in an  analogous way. \par
An operator $  \Lambda  \in \Ele(\pL(\eX))$ is a minimal solution of \eqref{eq58} if, for every elementary operator $\Lambda' \ne 0$ satisfying 
$\ell(\Lambda- \Lambda') + \ell(\Lambda')=\ell(\Lambda)$, one has $\Lambda' \Delta \neq 0$ (that is, $\varphi^{-1} (\Lambda)$ is a minimal tensor 
satisfying $\varphi^{-1} (\Lambda)\varphi^{-1} (\Delta)=0$).

\begin{lemma} \label{lem03}
Every solution of \eqref{eq58} is a sum of minimal solutions.
\end{lemma}

\begin{proof}
Let $\Lambda \neq 0$ be a solution of \eqref{eq58}. We proceed by induction on the length of $\Lambda$. The case $ \ell(\Lambda)=1$ is trivial. 
Assume therefore that $\ell(\Lambda) >1$ and that the desired conclusion holds for any $k < \ell(\Lambda)$. If $\Lambda$ is not a minimal 
solution of \eqref{eq58}, then there exists a nonzero elementary operator $\Lambda'$ such that $\ell(\Lambda)= \ell(\Lambda')+ \ell(\Lambda-\Lambda')$ 
and $\Lambda' \;\Delta=0$. Hence $(\Lambda-\Lambda')\; \Delta=0$. By the induction hypothesis we have that $\Lambda'$ and $\Lambda- \Lambda'$ are sums 
of minimal solutions of \eqref{eq58}. This yields the desired result.
\end{proof}

\begin{lemma}\label{rank1} 
Let $\eX_1, \eX_2,\eY_1, \eY_2$ be finite-dimensional complex vector spaces and $A \in \pL(\eX_1, \eY_1)$, $B \in  \pL(\eX_2, \eY_2)$. 
If $u \in \eX_1 \otimes \eX_2$ is a minimal tensor in the kernel of $A \otimes B$, then it is a simple tensor.
\end{lemma}

\begin{proof}
By Lemma \ref{dim}, either  $\Ker B \cap \sR(u) \neq \{ 0\}$ or $ \Ker A \cap \sL(u) \neq \{ 0\}$. Suppose, for instance, that there exists 
$y_1 \in \Ker B \cap \sR(u)$ such that $y_1\neq 0$. Write  $u= \sum_{i=1}^n x_i \otimes y_i$, where  $x_i \in \eX_1, y_i \in \eX_2$ and $n= \ell(u)$. 
Then $x_1 \otimes y_1 \in \Ker(A \otimes B)$. Therefore $n=1$.
\end{proof}

\begin{proposition} \label{prop04}
Let $A, B \in \pL(\eX)$ be nonzero operators.  The equation
\begin{equation} \label{eq111}
\Chi \Mu_{A,B}=0
\end{equation}
 has a nontrivial solution in $\Ele(\pL(\eX))$ if and only if
$A$ is a right zero divisor or $B$ is a left zero divisor. Moreover, any minimal solution of \eqref{eq111} in  $\Ele(\pL(\eX))$ has length one.
\end{proposition}

\begin{proof}
Assume that $\Delta$ is a minimal solution of \eqref{eq111}. Let $\widehat{B}$ and $\widetilde{A}$ denote the restriction of $\rL_B$ to $\sR(\Delta)$ 
and the restriction of  $\rR_A$ to $\sL(\Delta)$, respectively. Set $\Delta= \sum_{i=1}^n \Mu_{E_i, F_i}$. Then, by Lemma \ref{inj},
\begin{equation*}
(\widetilde{A} \otimes \widehat{B} )(\sum_{i=1}^n  E_i \otimes F_i)=0.
\end{equation*}
By Lemma \ref{rank1}, $\Delta$ has length one. Now it is easy to see that either $A$ is a right zero divisor or $B$ is a left zero divisor.
\end{proof}

\begin{remark} 
Let $\Delta \in \Ele(\pL(\eX)) $ be an elementary operator of arbitrary length. Suppose that there exists a multiplication operator $\Mu_{A, B}$ 
such that  $\Mu_{A,B} \Delta=0$. Then there exists $\Delta' \in \Ele(\pL(\eX))$ such that  $\ell(\Delta')+ \ell (\Delta- \Delta')= \ell (\Delta)$,  
$A \mathsf{L}(\Delta')=0$ and $ \mathsf{R}(\Delta- \Delta') B=0$.
\end{remark}

Now we consider elementary operators of length $2$. In \cite{AS}, the authors provide an interesting characterization of two-dimensional spaces
of complex matrices of bounded rank.  We will need a more detailed description. It should be pointed out that \cite[Corollary 1]{AS} can be deduced 
from our characterization.

\begin{lemma}\label{brank} 
Let $n, m,r\in \bN$ and let $\pS$ be a two-dimensional subspace in $\bM_{m, n}$, the space of all complex $m\times n$ matrices.
Suppose that  $1 \leq r \leq \min\{n,m\}$, $\rk(\pS)=n-r$, $\ker \pS= \{0\}$, and $\pS \bC^n= \bC^m$. Then $\pS$ is equivalent to the following space 
of  matrices
\begin{equation*}
\left\{  \left(   \begin{array}{ccccc}
A_{n_1}(\alpha, \beta) & 0 & \cdots & 0&0 \\
0 & \ddots & \ddots &\vdots& \vdots  \\
\vdots&\ddots&\ddots&0&\vdots\\
\vdots & \cdots &0& A_{n_r} (\alpha, \beta) & 0 \\
0 & \cdots & 0 &0&  * \\
\end{array} \right):\quad \alpha, \beta \in \bC \right\},
\end{equation*}
where
\begin{equation*}
A_{n_i} (\alpha, \beta)= \left(\begin{array}{ccccc}
                           \beta & \alpha & 0 &\cdots& 0 \\
                           0 & \ddots& \ddots & \ddots &\vdots\\
                           \vdots& 0 & \ddots & \ddots&\vdots \\
                           0 &\cdots  &\cdots& \beta & \alpha \\
                         \end{array}	
                       \right)  \in \bM_{n_i-1,n_i}.
\end{equation*}
\end{lemma}

\begin{proof} 
Choose  $B_1 \in \pS$ such that $\rk( B_1)=n-r$ and let $0 \neq y_1 \in \bC^n$ be such that $B_1y_1=0$. 
Pick $B_2 \in \pS \setminus \bC B_1$. We claim that there exists a subspace $\eW_1$
of $\bC^n$ containing $y_1$ such that $\pS \eW_1 \subseteq B_1 \eW_1$, space $\pS|_{\eW_1} $ has constant
rank, and $B\eW_1= B_1 \eW_1$, for all $B \in \pS$. 
Suppose, towards a contradiction, that there exists a family $\{y_1, \ldots, y_t\}$ of elements of $\bC^n$ such that
\begin{equation*}
B_1 y_1=0,\quad B_2 y_k= B_1 y_{k+1}\quad (1 \leq k \leq t-1) \quad \text{and} \quad  B_2y_t   \not\in  B_1 \bC^n,
\end{equation*}
with $t$ minimal ($t$ can be equal to $1$). Set $\eW'_1= \spa \{y_1, \ldots, y_t\}$. Observe that for every nonzero
complex number $\lambda$ one has $(B_1+ \lambda B_2) \eW'_1= B_1 \eW'_1+ \mathbb C B_2 y_t$.  Let $\eH$
be a subspace of $\bC^n$  of minimal dimension such that $B_1 \bC^n= B_1 \eW'_1 \oplus B_1 \eH$.  Using \cite[Lemma 2.1]{BrSe}
we see that one can  choose a nonzero $\lambda \in \bC$ such that $\rk(B_1+ \lambda B_2)|_\eH \geq \rk (B_1|_\eH)$ and 
$\{ (B_1+ \lambda B_2)y_2, \ldots, (B_1+ \lambda B_2)y_t, B_2 y_t\}$ is linearly independent modulo space $(B_1+ \lambda B_2)\eH$.   
Since
\begin{equation*}
(B_1+ \lambda B_2) \eW'_1= \spa \{B_1y_2, \ldots, B_1y_t, B_2 y_t\},
\end{equation*}
the rank of $B_1+ \lambda B_2$ is greatest than the rank of $B_1$, which is a contradiction.
Now suppose that we have constructed $y_2, \ldots, y_{t} \in \bC^n$ such that
\begin{equation*}
 B_2 y_k= B_1 y_{k+1} \quad \text{and} \quad B_2 y_{t} \in \spa \{B_2 y_1, \ldots, B_2 y_{t-1} \}\quad  (1 \leq k \leq t-1).
\end{equation*}
Set $\eW_1= \spa \{y_1, \ldots, y_t\}$. Then $\mathcal S \eW_1 \subseteq B_1 \eW_1 $.
 A straightforward computation shows that there exists a basis $ \{z_1, \ldots, z_t\}$ of $\eW_1$  such that
\begin{equation*}
B_1z_1=B_2 z_{t}=0 \quad \text{and}\quad B_2 z_k= B_1 z_{k+1},\quad \text{for}\quad 1 \leq k \leq t-1.
\end{equation*}
(Indeed, write  $B_2 y_{t}=\sum_{i=1}^{t-1} \alpha_i B_2 y_i$, where $\alpha_1, \ldots, \alpha_{t-1} \in \bC$, and put 
$z_1=y_1$ and $z_k=y_k-\sum_{i=t-k+1}^{t-1} \alpha_i y_{i-t+k}$, for $2 \leq k \leq t$).
Observe that for every nonzero $B \in \mathcal S$ we have $B_1 \eW_1= B \eW_1$. Then, with respect to the bases $\{z_1, \ldots, z_t\}$ 
and $\{B_1z_2, \ldots,B_1z_t\}$ the matrix  $B_2|_{\eW_1}$ has the form
\begin{equation*} 
\left( \begin{array}{cc}
         I_ {t-1} & 0 \\ 
\end{array} \right),
\end{equation*}
where $I_{t-1}$ is the identity matrix of order $t-1$. 
Therefore, for any complex numbers $\alpha, \beta$, the matrix $(\alpha B_1+ \beta B_2)|_{\eW_1}$ has form $A_{t}(\alpha, \beta)$. 
The claim is proved.

Next, using the above procedure, we construct a subspace $\eW=\oplus_{i=1}^r \eW_i$ of $\bC^n$, such that 
$B\eW_i =B_1\eW_i$, for every $B \in \mathcal S$, $\dim \eW_i=t_i$, and $\pS |_{\eW_i}$ has constant rank $t_i-1$. 
Indeed, suppose that we have constructed $l$ subspaces $\eW_1, \ldots, \eW_l$ of $\bC^n$ such that  
$\Dim(\eW_i)= t_i$, and $\eW_i= \spa\{z_1^i, \ldots, z_{t_i}^i\}$ with
\begin{equation*}
B_1 z_1^i=B_2 z_{t_i}^i=0\quad \text{and}\quad B_2 z_k^i= B_1 z_{k+1}^i, \quad \text{for}\quad 1 \leq k \leq t_i-1 .
\end{equation*}
Suppose that $\eW_1, \ldots, \eW_l$ are chosen such that $(t_1, \ldots, t_l)$ is minimal (with respect to the lexicographical order).
Set $\eW'= \eW_1 \oplus \cdots \oplus \eW_l$. Then, for every nonzero $B \in \pS$, one has $B_1 \eW'= B \eW'$. Suppose that $r \geq l+1$.
Pick  $z'_1 \in  \Ker B_1 \setminus \eW'$. Choose $z'_2, \ldots, z'_s \in \bC^n$ such that
\begin{equation*}
 B_2 z'_k= B_1 z'_{k+1} \quad (1 \leq k \leq s-1) \quad \text {and }  \quad B_2 z'_s  \in \spa \{B_2 z'_1, \ldots, B_2 z'_{s-1} \} + B_2 \eW'.
\end{equation*}
For $w = \sum \alpha_i^j z_i ^j \in \eW'$, set $w_{-h}= \sum \alpha_i^j z_{i-h}^j$, where $z_{i-h}^j=0$ if $h \geq i$.
Write  $B_2 z'_s=\sum_{i=1}^{s-1} \alpha_i B_2 z'_i+ B_2 w$, where $\alpha_1, \ldots, \alpha_{s-1} \in \bC$ and $w \in \eW'$.  
Observe that we can suppose that $w \in \sum_{j=1}^l \spa \{z_1^j, \ldots, z_{t_j-1}^j \}$. Put $z_1^{l+1}= z'_1$, 
$z_k^{l+1}= z'_k-\sum_{i=s-k+1}^{s-1} \alpha_i z'_{i-s+k}-w_{-s+k}$, for $2 \leq k \leq s$, $s=t_{l+1}$ and 
$\eW_{l+1}= \spa \{ z_1^{l+1}, \ldots, z _{t_{l+1}}^{l+1}  \} $. Then
\begin{equation*}
B_1 z_1^{l+1}=B_2 z_{t_{l+1}}^{l+1}=0 \quad \text{and}\quad B_2 z_k^{l+1}= B_1 z_{k+1}^{l+1}\quad \text{for} \quad 1 \leq k \leq t_{l+1}-1.
\end{equation*}
Using this process, we construct $\eW$. Observe again that, for every nonzero $B \in \pS$,  $B_1 \eW= B \eW$.  Moreover, for every $1 \leq i \leq r$, 
the space $\pS|_{\eW_i}$ is equivalent to $\{A_{t_i} (\alpha, \beta): \alpha, \beta \in \bC \}.$

Now let $\eH$ be a subspace of $\bC^n$ of minimal dimension such that  $B_1 \bC^n \oplus B_2 \eH= \pS \bC^n$. Write $\bC^n= \eW \oplus \eH \oplus \eZ'$. 
Clearly, we can assume that $B_2 \eZ' \subseteq B_1 \bC^n$. Observe that $B_{1}|_{\eZ'}$ is injective.  Let $S: B_1 \eZ' \rightarrow  \eZ'$ be linear 
such that $SB_{1}|_{\eZ'}=I$. Let $\tau: \pS \bC^n \rightarrow B_1 \eZ' $ be the natural projection. Choose a basis $\{z'_1, \ldots,  z'_s\}$  of $\eZ'$ 
such that the matrix representation of $\pS \tau B_2|_{\eZ'}$ is upper triangular.  There exists $\lambda \in \bC$ such that $\tau B_2 z'_1 = \lambda B_1 z'_1$. 
Hence $B_2 z'_1=  \lambda B_1 z'_1+ B_1 v'_1 +B_1 \mu_1$, where $v'_1 \in \eW$ and $\mu_1 \in \eH$. Choose $v_1 \in \eW$ such that 
$B_1 v'_1= (\lambda B_1 - B_2) v_1$ and put $z_1= z'_1+ v_1$. Then $B_2 z_1= \lambda B_1  z_1+B_1 \mu_1$.  Suppose we have constructed 
$z_2, \ldots, z_k$ such that $z_i= z'_i+ v_i$, $v_i \in \eW$ and $B_2 z_i \in B_1 \;(\eH \oplus \spa \{ z_1, \ldots, z_i\})$, for $2 \leq i \leq k$.
Write $\tau B_2 z'_{k+1}= \sum_{i=1}^{k+1} \alpha_i B_1 z'_i$. Then there exists $v'_{k+1} \in \eW, \mu_{k+1} \in \eH$ such that
\begin{equation*}
 B_2 z'_{k+1}= \sum_{i=1}^{k} \alpha_i B_1 z_i + \alpha_{k+1}B_1 z'_{k+1}+ B_1 (v'_{k+1} - \sum_{i=1}^{k} \alpha_i v_i)+ B_1 \mu_{k+1}.
\end{equation*}
Let $v_{k+1} \in \eW$ be such that $B_1 (v'_{k+1} - \sum_{i=1}^{k} \alpha_i v_i)= (\alpha_{k+1} B_1 - B_2) v_{k+1}$. Put $z_{k+1}= z'_{k+1}+ v_{k+1}$.  
Then $B_2 z_{k+1} = \sum_{i=1}^{k+1} \alpha_i B_1 z_i+ B_1 \mu_{k+1}$. We have thereby shown that there exists a subspace $\eZ$ of $\bC^n$ such that 
$\bC^n= \eW  \oplus \eH\oplus \eZ$ and   $B_2 \eZ \subseteq B_1 (\eZ\oplus \eH)$. Now the desired conclusion follows easily.
\end{proof}

\begin{proposition}\label{rank2} 
Let $\eX_1, \eX_2, \eY_1, \eY_2$ be finite-dimensional complex vector spaces and let $\Delta \in \pL(\eX_1, \eY_1)\otimes \pL(\eX_2, \eY_2)$ be 
a tensor of rank $2$. Suppose that $u \in \Ker \Delta$ is a minimal tensor of rank $n \geq 2$. Then there exist simple tensors 
$\Mu_1, \Mu_2 \in \pL(\eX_1, \eY_1)\otimes \pL(\eX_2, \eY_2)$ and simple tensors $u_1, \ldots, u_n  \in \eX_1 \otimes \eX_2$ such that 
$\Delta= \Mu_1+ \Mu_2$, $u= u_1+ \cdots + u_n$,
\begin{equation*}
\Mu_2 u_1 = \Mu_1u_n=0\quad \text{and}\quad \Mu_1u_k+ \Mu_2u_{k+1}=0\qquad (1 \leq k \leq n-1).
\end{equation*}
\end{proposition}

\begin{proof}
Let $\Delta= \sum_{i=1}^2 A_i \otimes B_i$ and $u= \sum_{i=1}^n x_i \otimes y_i$. Then $\sum_{j=1}^n \sum_{i=1}^2 A_i x_j \otimes B_i y_j =0.$
Hence, by Lemma \ref{dim}, 
\begin{equation*}
\Dim(\spa \{ A_ix_j: 1 \leq i \leq 2, 1 \leq j \leq n \}) + \Dim(\spa \{ B_iy_j: 1 \leq i \leq 2, 1 \leq j \leq n \}) \leq 2n.
\end{equation*}
Therefore, either $\Dim(\sL(\Delta) \sL(u)) \leq n$ or $\Dim(\sR(\Delta)\sR(u)) \leq n$. Suppose, for instance, that $\Dim(\sR(\Delta)\sR(u)) \leq n$. 
For each $B \in \sR(\Delta)$, let $\widehat{B}: \sR(u) \rightarrow \sR(\Delta) \sR(u)$ denote the restriction of  $B$ to $\sR(u)$ and set 
$\pS=\{\widehat{B}: B \in \sR(\Delta)\}$. Since $u$ is a minimal element of $\Ker \Delta$, then $\Ker \, \pS= \{0\}$, (otherwise, set 
$u= \sum_{i=1}^n x_i \otimes y_i$, where $\pS y_1=0$; then $\Delta (x_1 \otimes y_1)=0$, a contradiction). Choose $B_1 \in \sR(\Delta)$ with the property 
that $\rk(B_1)= \rk(\pS)$. We distinguish two cases.

Case 1. Suppose first that $\rk(\pS)=n$. Then $\Dim(\sR(\Delta)\sR(u))=n$ and $\widehat{B}_1$ is bijective. Pick a nonzero element $ B_2 \in \sR(\Delta)$ 
which is not injective. Write $\Delta= \sum_{i=1}^2 \Mu_{A_i, B_i}$ for suitable $A_1, A_2 \in \sL(\Delta)$. Choose a Jordan basis $\{y_1, \ldots, y_n\}$  
for $\widehat{B}_1^{-1}\widehat{B}_2$ and suppose that $\{y_1, \ldots, y_{i_1}\}$ is associated to the first block, with $\widehat{B}_1^{-1}\widehat{B}_2 y_1=0$. 
Then 
\begin{equation*} 
B_2 y_1=0,\qquad B_2 y_{k+1}= B_1 y_k, \quad \text { for } 1 \leq k \leq i_1-1,  
\end{equation*}  
and 
\begin{equation*} 
\widehat{B}_1^{-1}\widehat{B}_2 (\spa \{y_{i_1+1}, \ldots, y_n\}) \subseteq \spa \{y_{i_1+1}, \ldots, y_n\}.
\end{equation*}
Write $u= \sum_{i=1}^n x_i \otimes y_i$ for  suitable $x_1, \ldots, x_n \in \mathsf{L}(u)$. We have
\begin{equation*}
\sum_{k=1}^{i_1-1} (A_2 x_{k+1} + A_1 x_k ) \otimes B_1 y_k  + A_1 x_{i_1} \otimes B_1 y_{i_1} + \sum_{\substack{i+1 \leq k \leq n \\  1\leq j\leq 2}}  A_j x_k  \otimes B_jy_k  = 0.
\end{equation*}
Since $B_2 y_k \in \spa \{B_1y_{i_1+1}, \ldots, B_1 y_n\}$ for all $k \geq i_1+1$ and $\widehat{B}_1$ is injective, we have
\begin{equation*}
\sum_{k=1}^{i_1-1} (A_2 x_{k+1} + A_1 x_k ) \otimes B_1 y_k  + A_1 x_{i_1} \otimes B_1 y_{i_1} =\sum_{\substack{i+1 \leq k \leq n \\  1\leq j\leq 2}}  A_j x_k  \otimes B_jy_k  = 0.
\end{equation*}
But $u$ is a minimal tensor in $\ker \Delta$, which gives $i_1=n$. Moreover, we have
\begin{equation*}
A_1x_n = B_2 y_1=0,\quad B_2 y_k=B_1 y_{k -1}\quad \text {and }\quad A_2x_k= -A_1 x_{k-1} \qquad  (2 \leq k \leq n).
\end{equation*}

Case 2. Suppose that $\rk(\pS) \leq n-1$. Set $\rk(\pS)= n-r$, where $1 \leq r \leq n-1$. Pick $B_2 \in \sR(\Delta) \setminus \bC B_1$.  
By Lemma \ref{brank}, there exist subspaces  $\eW_1, \ldots, \eW_r, \eZ$ of $\sR(u)$ such that 
$B_2 \eW_i = B_1 \eW_i$, $\pS \eZ \cap  \pS (\eW_1 + \cdots \eW_r)= \{ 0\}$ and $\pS \eW_i \cap \pS \eW_j= \{ 0\}$, for $i \neq j$. 
Since $u$ is a minimal element of $\Ker \Delta$ it has to be $\eZ= \{ 0\}$. The same argument implies that $\sR(u)= \eW_1$. We have thereby shown that
$\sR(u)= \spa\{y_1, \ldots, y_n\}$, where
\begin{equation*}
B_1y_1=B_2y_n=0  \quad \text {and } \quad B_2y_k= B_1 y_{k+1}, \; (1 \leq k \leq n-1).
\end{equation*}
Write $u= \sum x_i \otimes y_i$ and $\Delta= \sum_{i=1}^2 \Mu_{A_i, B_i}$. Then one has $A_2 x_k= -A_1 x_{k+1}$, for $1 \leq k \leq n-1$. 
Set $u_i= x_{n+1-i} \otimes y_{n+1-i}$. This yields the desired result.
\end{proof}

\begin{theorem}\label{zerodiv} 
Let $\Psi \in \Ele(\pL(\eX))$ be an elementary operator of length $2$. Suppose that the equation $\Chi \Psi=0$ has a minimal solution  
$\Phi \in \Ele(\pL(\eX))$ of length $n\geq 2$. Then there exist multiplication operators  $\Gamma_1, \ldots, \Gamma_n, \Mu_1, \Mu_2$ such that 
$ \Psi=\Mu_1+ \Mu_2$, $\Phi= \Gamma_1+ \cdots+ \Gamma_n$, and
\begin{equation*}
\Gamma_1\, \Mu_2= \Gamma_n \,\Mu_1=0, \quad \Gamma_k \,\Mu_1+ \Gamma_{k+1}\,\Mu_2=0 \qquad (1 \leq k \leq n-1).
\end{equation*}
\end{theorem}

\begin{proof}
Write  $\Psi= \sum_{i=1}^2 \Mu_{A_i, B_i}$ and $\Phi= \sum_{i=1}^n \Mu_{E_i, F_i}$. It follows, by Lemma \ref{inj}, that 
\begin{equation*}
\sum_{i=1}^2 \sum_{j=1}^n E_jA_i \otimes B_i F_j=0,
\end{equation*}
which yields
\begin{equation*}
(\sum_{i=1}^2 R_{A_i} \otimes L_{B_i}) (\sum_{j=1}^n E_j \otimes F_j)= 0.
\end{equation*}
Let $\eX_1=\sL(\Phi)$, $\eX_2= \sR(\Phi)$, $\eY_1=\sL(\Phi) \sL(\Psi)$, and $\eY_2= \sR(\Psi) \sR(\Phi)$. 
Then $\Phi$ is a minimal tensor in $\Ker (\sum_{i=1}^2 \rR_{A_i} \otimes \rL_{B_i})$.  Consequently, the desired conclusion follows from
Proposition \ref{rank2}.
\end{proof}

Let $T$ be an operator on $\eX$.  If $T$ is algebraic,  we denote the degree of its minimal polynomial by $\deg(T)$. For a non-algebraic operator $T$
we set $\deg (T)= \infty$.

\begin{corollary} 
Let $A, B \in \pL(\eX)$. Suppose that $\ell(\Upsilon_{A,B})=2$. Then the equation
\begin{equation} \label{eq333}
\Chi \Upsilon_{A,B}=0
\end{equation} 
admits a solution in $\Ele(\pL(\eX))$ if and only if there exist a nonzero complex number  $\lambda$  and $n$-dimensional subspaces $\eU$ and $\eV$  
of $ \pL(\eX)$ such that $A \eU \subseteq \eU$, $B \eV \subseteq \eV$, and $\rR_{I  + \lambda A}|_{\eU} $, $\rL_{B- \lambda I}|_{\eV}$ are nilpotent
of degree $n$. Moreover, in this case, there exists a multiplication operator $\Mu$ such that $\Mu \Upsilon_{A,B}=0$.
\end{corollary}

\begin{proof}  Let $\Delta$ be a minimal solution of \eqref{eq333} of length $n$. Observe that $I \in \mathsf{R}(\Upsilon_{A,B}) \cap \mathsf{L} (\Upsilon_{A,B})$. 
Hence $\Dim(\sR(\Upsilon_{A,B})\sR(\Delta)) \geq n$ and $\Dim(\sL(\Delta) \sL( \Upsilon_{A,B}))\geq n$.
By Lemma \ref{dim}, $\Dim(\sR(\Upsilon_{A,B})\sR(\Delta))=n$. It follows from the proof of Proposition \ref{rank2} 
that there exist $\lambda \in \bC$ and a representation of $\Delta$ as $\Delta= \sum_{i=1}^n \Mu_{E_i, F_i}$ such that
$(B- \lambda I ) F_1=E_n(I+ \lambda A) = 0$ and
\begin{equation*}
(B- \lambda I)F_k= F_{k-1},  \quad E_k A =- E_{k-1}(I+ \lambda A)  \qquad (2 \leq k \leq n).
\end{equation*}
Consequently, $\Mu_{E_n, F_1}  \Upsilon_{A,B} =0$ and $\lambda \neq 0$. Moreover, a straightforward computation shows that the restriction of 
$\rL_{B- \lambda I}$ to $\spa \{F_1, \ldots, F_n\}$ is nilpotent of degree $n$. Similarly, the restriction of $\rR_{I+ \lambda A}$ to 
$\spa \{E_1, \ldots, E_n\}$ is nilpotent of degree $n$, as well. 

Conversely, suppose that there exist a nonzero $\lambda \in \bC$ and subspaces $\eU, \eV$ of $\pL(\eX)$ of dimensions $n$ such that $A \eU \subseteq \eU$, 
$B \eV \subseteq \eV$ and $\rR_{I+ \lambda A}|_{\eU}$ and $\rL_{B- \lambda I}|_{\eV}$ are nilpotents of degree $n$. 
Choose $F_n \in \eV$ such that the set $\{(B- \lambda I) F_n, \ldots, (B- \lambda I)^{n-1} F_n\}$ is linearly independent. Set $F_{k-1}= (B- \lambda I) F_k$ 
for $2 \leq k \leq n$. Then $(B- \lambda I )F_1=0$. Next, choose $E_1 \in \eU$  such that the set $\{E_1(I+ \lambda A), \ldots, E_1(I+ \lambda A)^{n-1}\}$ is 
linearly independent. Since $\rR_{(I+ \lambda A)^n}|_{\eU} = 0$,  operator $\rR_A|_{\eU}$ has to be invertible. Hence we can construct operators 
$E_2, \ldots, E_n\in \eU$ such that $E_k(I+ \lambda A) = - E_{k+1} A$ for $k=1, \ldots, n-1$. Since $E_1 (I+ \lambda A)^n =0$ and $\rR_A|_{\eU}$ is invertible, we get 
$E_n (I+ \lambda A)=0$. Write $\Upsilon_{A,B}= \Mu_{I+ \lambda A, I}+ \Mu_{A, B- \lambda I}$. A straightforward computation shows that 
$\sum_{i=1}^n \Mu_{E_i, F_i}\Upsilon_{A,B}=0$. 
\end{proof}

\begin{theorem} \label{M1} 
Let $\Psi \in \Ele(\pL(\eX))$ be an elementary operator of length $2$. The equation $\Chi \Psi=0$ has a solution in $ \Ele(\pL(\eX))$ if and only if
one of the following conditions holds.

{\rm (1)} There exists a multiplication operator $\Mu \in \Ele(\pL(\eX))$ such that $\Mu \Psi=0$.

{\rm (2)} There exist two vector subspaces $\eU, \eV$ of $\pL(\eX)$, each of dimension $n$, such that the space $\{\rL_B|_{\eV}: B \in \sR(\Psi)\}$ is equivalent 
to a constant rank  $n-1$  subspace of $ \bM_{n-1, n}$ and the space $\{\rR_A|_{\eU}: A \in \sL(\Psi)\}$ is equivalent to a constant rank $n$ subspace of 
$\bM_{n+1, n}$. Moreover, $\Rann(\sR(\Psi)) = \{0\}$ and $\rR_A$ is injective for all $A \in \sL(\Psi)\setminus\{0\}$.

{\rm (3)} There exist two vector subspaces $\eU, \eV$ of $\pL(\eX)$ of dimension $n$ such that the space $\{\rL_B|_{\eV}: B \in \sR(\Psi)\}$ is equivalent to a  
constant rank $n$ subspace of $\bM_{n+1, n}$ and the space $\{\rR_A|_{\eU}: A \in \sL(\Psi)\}$ is equivalent to a constant rank $n-1$ subspace of 
$ \bM_{n-1, n}$. Moreover, $\Lann(\sL(\Psi)) = \{0\}$ and $\rL_B$ is injective for all $B \in \sR(\Psi)\setminus\{0\}$.
\end{theorem}

\begin{proof} 
Suppose that there exists an elementary operator $\Delta$ of length $n$ such that $\Delta \Psi=0$. Using once again Lemma \ref{dim}, we see that 
either $\Dim(\sR(\Psi)\sR(\Delta)) \leq n$ or $\Dim(\sL(\Delta) \sL(\Psi)) \leq n$. Suppose first that $\Dim(\sR(\Psi)\sR(\Delta)) \leq n$. 
A careful reading of the proof of Proposition \ref{rank2} and Theorem \ref{zerodiv} shows that we have only to treat the case where the space 
$\{\rL_B|_{\sR(\Delta)}: B \in \sR(\Psi) \}$ has constant rank $ n-1$ and  $\Dim(\sR(\Psi) \sR(\Delta)) = n-1$ (indeed, if the space 
$\{\rL_B|_{\sR(\Delta)}: B \in \sR(\Psi) \}$ has rank $n$, then there exists a multiplication operator $\Mu$ such that $\Mu \Psi=0$). It follows,
by Lemma \ref{inj}, that $\Dim(\sL(\Delta) \sL(\Psi)) \leq n+1$. If $\Rann(\sR(\Psi)) \neq \{ 0\}$, then it is obvious that there exists a multiplication 
operator $\Mu$ such that $\Mu \Psi=0$. Next suppose that there exists $A \in \sL(\Psi)\setminus \{ 0\}$ such that $\rR_A$ is not injective. 
Write $\Psi= \Mu_{A,B}+ \Mu_{A', B'}$ for suitable $A' \in \sL(\Psi)$ and $B, B' \in \sR(\Psi)$ (Lemma \ref{lem02}). Choose 
$E \in \pL(\eX)$, $F \in \sR(\Delta)$ such that $EA=B'F=0$. Then $\Mu_{E,F} \Psi=0$. Now suppose that the space $ \{\rR_A|_{\sL(\Delta)}: A \in \sL(\Psi)\}$ 
has constant rank $n$. Since the constant rank $n$ subspaces of $\bM_n$ have dimension $1$ we conclude that $\Dim(\sL(\Delta) \sL(\Psi)) = n+1$, as desired. 
The case $\Dim(\sL(\Delta) \sL(\Psi)) \leq n$ is treated similarly.

For the converse, suppose that $(2)$ holds. Arguing as in the proof of Lemma \ref{brank}, we see that the set $\{ \rL_B|_{\eV}: B \in \sR(\Psi)\}$ is 
equivalent to the constant rank subspace of $\bM_{n-1, n}$ of the form
\begin{equation*}
\left( \begin{array}{ccccc}
    \beta & \alpha & 0& \cdots & 0 \\
    0 & \beta & \alpha & \cdots & 0 \\
    \vdots & \vdots & \ddots & \ddots & \vdots \\
    0 & 0 & \cdots & \beta & \alpha \\
\end{array} \right).
\end{equation*}
On the other hand, it is easy to show that the set $\{ \rR_A|_{\eU}: A \in \sL(\Psi)\}$ is equivalent to the constant rank subspace of $\bM_{n+1, n}$ 
of the form

\begin{equation*}
\left( \begin{array}{ccccc}
\alpha & 0 & \cdots & 0 & 0 \\
\beta & \alpha & \cdots & 0 & 0 \\
0 & \beta & \cdots& 0 & 0 \\
\vdots & \vdots & \ddots & \ddots & \vdots \\
0 & 0 & \cdots & \beta & \alpha \\
0 & 0 & \cdots & 0 & \beta \\
\end{array} \right)
\end{equation*}
Write $\Psi= \sum_{i=1}^2 \Mu_{A_i, B_i}$, where $A_i \in \sL(\Psi)$ and $B_i \in \sR(\Psi)$. Then there exist $E_1, \ldots, E_n \in \eU$ and 
$F_1, \ldots, F_n \in \eV$ such that
\begin{equation*}
B_1 F_1= B_2 F_n=0,\quad B_2 F_{k}= B_1 F_{k+1},\quad \text { and }\quad E_k A_2  = -E_{k+1}A_1 \qquad (1 \leq k \leq n-1).
\end{equation*}
It is easily seen that $\Delta =\sum_{i=1}^n \Mu_{E_i, F_i}$ is a left zero divisor of $\Psi$. 
The case (3) is treated similarly.
\end{proof}

\begin{remark} 
Let $\Psi \in \Ele(\pL(\eX))$ be an elementary operator of length $2$. Suppose that the equation $\Chi \Psi=0$ has a  solution  in 
$ \Ele(\pL(\eX))$ and $\Mu \Psi \neq 0$, for every nonzero multiplication operator $\Mu$. Then, for every elementary operator $\Phi$ satisfying 
$\sL(\Phi)= \sL(\Psi)$ and $\sR(\Phi)=\sR(\Psi)$, the equation $\Chi \Phi=0$ admits a solution  in $ \Ele(\pL(\eX))$.
\end{remark}

\section{Invertibility} \label{sec04}
\setcounter{theorem}{0}

In this section we are concerned with the (left, respectively right) invertibility in the algebra $\Ele(\pL(\eX))$ of short elementary operators.
Thus, the main question is, under which condition on an invertible elementary operator $\Delta$ of length $1$ or $2$ does the equation
$\Delta \Chi=\Iota$, respectively the equation $\Chi \Delta=\Iota$, or the system of both,
have a solution in $\Ele(\pL(\eX))$? \par
Recall from \cite[Ch. II, Theorem 16]{Mul} that a bounded linear operator on a Banach space is left-invertible if and only if it is bounded below and its
range is a complemented subspace. Similarly, a bounded linear operator on a Banach space is right-invertible if and only if it is surjective and its kernel
is a complemented subspace.

It is easily seen that a two sided multiplication operator $\Mu_{A,B}$ is invertible if and only if $A, B\in \pL(\eX)$ are
invertible. In this case, $\Mu_{A,B}^{-1}=\Mu_{A^{-1},B^{-1}}$. What about the existence of the right (respectively, left) inverse of
$\Mu_{A,B}$ in $\Ele(\pL(\eX))$?

\begin{proposition} \label{prop03}
For $A, B \in \pL(\eX)$, the equation
\begin{equation} \label{eq46}
\Chi\Mu_{A,B}=I
\end{equation}
has a solution in $\Ele(\pL(\eX))$ if and only if $A$ is right-invertible and $B$ is left-invertible. Moreover, any minimal solution of \eqref{eq46} 
has length one.
\end{proposition}

\begin{proof}
If $A$ is right-invertible with a right inverse $A_r$ and $B$ is left-invertible with a left inverse $B_l$, then $\Mu_{A_r, B_l}$ solves the equation 
\eqref{eq46}. For the opposite implication, if there exists a multiplication operator which solves \eqref{eq46}, then we are done. Assume that 
$\Lambda \in \Ele(\pL(\eX))$ is a minimal solution of \eqref{eq46} of length $n \geq 2$.
Write $B\sR(\Lambda)=\bC I \oplus \pW$, for some suitable subspace $\pW$ of $B \sR(\Lambda)$. Let $\pi: B\sR(\Lambda) \rightarrow \pW$ be the 
natural projection. Then $\Lambda \Mu_{A, \pi(B)}=0$. It follows, by Proposition \ref{prop04}, that one can write $\Lambda=\sum_{i=1}^n \Mu_{E_i, F_i}$ 
such that either $E_i A=0$ or $\pi(B) F_i=0$. Since $\Lambda$ is a minimal solution of \eqref{eq46} only the last case is possible. Thus,  
$\pi(B)F_i=0$, for any $i$. Choose $i$ such that $B F_i \neq 0$. Clearly, we can assume that $BF_i= I$. Let $j \neq i$, with $1 \leq j \leq n$. Then there exists 
$\lambda \in \bC$ such that $B(F_j- \lambda F_i)=0$. Write $\Lambda= \Mu_{E_j, F_j- \lambda F_i }+ \Delta$, where $\ell(\Delta)= n-1$. Then 
$\Mu_{E_j, F_j- \lambda F_i } \Mu_{A,B}=0$, and consequently, $\Delta \Mu_{A,B}=\Iota$, a contradiction. This completes the proof.
\end{proof}

Now we consider elementary operators of length $2$. We start with length $2$ elementary operators of the form $\Upsilon_{A,B}$, where $A, B \in \pL(\eX)$.  
Let $\Tau_{A,B}= \rL_A-\rR_B$ be the generalized derivation implemented by $A$ and $B$. Note that for a suitable scalar $\lambda$, operator 
$\rL_{(A- \lambda I)^{-1}} \Tau_{A,B} $ has the form $\Upsilon_{C, D}$, for some $C,D \in \pL(\eX)$. On the other hand, 
$\Upsilon_{A,B}=\rR_B\Tau_{A,-B^{-1}}$ whenever $B$ is invertible. In \cite{Ros}, Rosenblum studied the inverse of an invertible generalized derivation 
$\Tau_{A,B}$.

\begin{lemma}\label{formula}
Let $\Tau_{A,B}\in \Ele(\pL(\eX))$ be an invertible generalized derivation of length $2$. Suppose that $B$ is algebraic of degree $n$ and
that $\{I, A, \ldots, A^{n-1}\}$ is a linearly independent set of operators.
Then the inverse $\Tau_{A,B}^{-1}$ is an elementary operator of length $n$.
\end{lemma}

\begin{proof}
Since $B$ is not a scalar multiple of $I$ the integer $n$ is actually at least $2$. Let
\begin{equation*}
m_B(z)=z^n+\alpha_{n-1} z^{n-1}+\cdots+\alpha_1 z+\alpha_0 
\end{equation*}
be the minimal polynomial of $B$. Then $m_B(A)$ is an invertible operator (since  $\sigma(B)\cap \sigma (A)= \emptyset$).
Let $\Delta'\in \Ele(\pL(\eX))$ be defined by
\begin{equation*}
\Delta'=  \sum_{i=1}^{n-2} \Mu_{A^{n-i} + \alpha_{n-1} A^{n-i-1}+ \cdots+ \alpha_{i+1} A,  B^{i-1}} +\Mu_{ A,B^{n-2}}
+\rR_{B^{n-1}+ \alpha_{n-1} B^{n-2}+ \cdots+ \alpha_2 B + \alpha_1 I}
\end{equation*}
(if $n=2$, then $\Delta'=\rL_A+\rR_{B+\alpha_1 I}$).
Observe that $\ell(\Delta')=n$. A straightforward computation shows that
$\Delta'\Tau_{A,B}=\Tau_{A,B}\Delta' = \rL_{m_B(A)}$.
Since $m_B(A)$ is invertible we see that $\Tau_{A,B}^{-1}=  \Delta' \rL_{m_B(A)^{-1}}=\rL_{m_B(A)^{-1}}\Delta'$
which means that $\Tau_{A,B}^{-1}$ is an elementary operator of length $n$.
\end{proof}

\begin{theorem}\label{standard} 
Let $A, B \in \pL(\eX) $. Suppose that $ \Upsilon_{A,B}$ is invertible and of length $2$.
Then $\Upsilon_{A,B}^{-1}$ is an elementary operator if and only if either $A$ or $B$ is algebraic.
Moreover, $\ell( \Upsilon_{A,B}^{-1}) = \min \{\deg(A),\deg(B)\}=n$.
\end{theorem}

\begin{proof} Suppose that $\Upsilon_{A,B}^{-1}$ is an elementary operator of length $n$.  By Lemma \ref{dim},
\begin{equation*}
\Dim(\sR(\Upsilon_{A,B}) \sR(\Upsilon_{A,B}^{-1}))+ \Dim(\sL(\Upsilon_{A,B}^{-1}) \sL( \Upsilon_{A,B})) \leq 2n+1.
\end{equation*}
Hence, either $\Dim(\sR(\Upsilon_{A,B}) \sR(\Upsilon_{A,B}^{-1})) \leq n$ or $\Dim(\sL(\Upsilon_{A,B}^{-1}) \sL( \Upsilon_{A,B})) \leq n$.
Suppose, for instance, that $\Dim(\sL(\Upsilon_{A,B}^{-1})\sL(\Upsilon_{A,B}))  \leq n$.  Since $I \in \sL(\Upsilon_{A,B})$  and
$\Dim(\sL(\Upsilon_{A,B}^{-1}))=n$ we have $ \sL(\Upsilon_{A,B}^{-1})= \sL(\Upsilon_{A,B}^{-1}) \sL( \Upsilon_{A,B})$.
Therefore, $\bC I+ \sL( \Upsilon_{A,B}^{-1}) A \subseteq \sL( \Upsilon_{A,B}^{-1})$. Let $ \sL( \Upsilon_{A,B}^{-1})= \spa \{E_1, \ldots, E_n\}$,
where $I= E_1$ and $E_k A= A^k= E_{k+1}$, for $1 \leq k \leq n-1$. Then $E_n A= A^n \in \spa \{I, A, \ldots, A^{n-1} \}$. As a result, 
$A$ is algebraic of degree at most $n$. We have thereby shown that either $A$ or $B$ is algebraic of degree at most $n$.

Suppose now that $A$ is algebraic. Since $\sigma (A)$ is finite we can choose $\lambda \in \bC$ such that $B- \lambda I$ and $I+ \lambda A$ are invertible. 
Write $\Upsilon_{A,B}= \Mu_{A, B- \lambda I}+ \rL_{I+ \lambda A}$. Then $\Mu_{(I+ \lambda A)^{-1}, (B- \lambda I)^{-1}} \Upsilon_{A,B}$ is a 
generalized derivation. The desired conclusion follows by Lemma \ref{formula}.
\end{proof}

\begin{corollary} \label{theo01}
Let $\Tau_{A,B}\in \Ele(\pL(\eX))$ be an invertible generalized derivation of length $2$.
The inverse $\Tau_{A,B}^{-1}$ is an elementary operator if and only if either $A$ or $B$ is algebraic. 
Moreover, $\ell( \Tau_{A,B}^{-1}) = \min \{\deg(A),\deg(B)\}=n$.
\end{corollary}

\begin{proof}
With no loss of generality, we may assume that $A$ is invertible (since we can replace $A$ and $B$
by $A- \lambda I$ and $B-\lambda I$ respectively). Then $\rL_{A^{-1}} \Tau_{A,B}=  \Upsilon_{-A^{-1},B}$.
Now the desired conclusion follows by Theorem \ref{standard}.
\end{proof}

Next we characterize generalized derivations whose inverses are generalized derivations, too.

\begin{corollary} \label{cor02}
Let $\Tau_{A,B}$ be an invertible generalized derivation of length $2$.
The inverse $\Tau_{A,B}^{-1}$ is a generalized derivation if and only if there exists
$ \lambda \in \bC$ such that $(A- \lambda I)^2$ and  $(B- \lambda I)^2$ are scalar multiples of $I$.
\end{corollary}

\begin{proof}
If there exists $\lambda \in \bC$ such that  $(A- \lambda I)^2$ and  $(B- \lambda I)^2$ are scalar multiples of $I$, say $(A- \lambda I)^2=\alpha I$ 
and $(B- \lambda I)^2=\beta I$, then $\alpha \ne \beta$ as $\Tau_{A,B}$ is invertible. It is easy to check that 
$\Tau_{A,B}^{-1}=\tfrac{1}{\alpha-\beta}\Tau_{A-\lambda I,-B+\lambda I}$.

Suppose now that $\Tau_{A,B}^{-1}$ is a generalized derivation, say $\Tau_{A,B}^{-1}=\Tau_{C,D}$. By Corollary \ref{theo01}, $A$ or $B$ is an algebraic
operator and $\min \{\deg(A),\deg(B)\}=2$.
Assume that $B$ is algebraic of degree $2$. There is no loss of generality if we assume that $B^2=\beta I$ for some $\beta \in \bC$
(we can replace $A$ and $B$ by $A-\lambda I$ and $B-\lambda I$, respectively, if necessary). Thus, the minimal polynomial of $B$ is $m_B(z)=z^2-\beta.$
By the proof of Lemma \ref{formula},
$\Tau_{A,B}^{-1}=L_{(A^2-\beta I)^{-1}A}+\Mu_{(A^2-\beta I)^{-1},B}$
and therefore, because of $\Tau_{A,B}^{-1}=\Tau_{C,D}$, one has
$\rL_{(A^2-\beta I)^{-1}A-C}+\Mu_{(A^2-\beta I)^{-1},B}+\rR_D=0$.
If $(A^2-\beta I)^{-1}$ and $I$ were linearly independent, then, by Corollary \ref{cor01}, $B$ and $D$ would be scalar multiples of $I$. As this is not the
case we conclude that $(A^2-\beta I)^{-1}$ is a scalar multiple of $I$. It is obvious now that $A^2=\alpha I$ for some $\alpha \in \bC$.
\end{proof}

\begin{corollary} \label{cor03}
Let $\Delta=\Mu_{A,B}+\Mu_{C,D}$  be an invertible elementary operator of length $2$.
If $B$ and $C$ are invertible, then the inverse $\Delta^{-1}$ is an elementary operator of length  $n$
if and only if $C^{-1}A$ or $DB^{-1}$ is an algebraic operator  and $\min \{\deg(C^{-1}A),\deg(DB^{-1})\}=n$.
\end{corollary}

\begin{proof}
Since $B$ and $C$ are invertible we have
$ \Delta=\Mu_{C,B}\Tau_{C^{-1}A, -DB^{-1}},$
which gives
$ \Delta^{-1}=\Tau_{C^{-1}A, -DB^{-1}}^{-1}\Mu_{C,B}^{-1}$. Hence, $\Delta^{-1}$ is an elementary operator of length $n$ if and only
$\Tau_{C^{-1}A, -DB^{-1}}^{-1}$ is an elementary operator of length $n$. Now use Corollary \ref{theo01}.
\end{proof}

\begin{lemma}\label{suminv} 
Let $A, B \in \pL(\eX)$. Suppose that $\Upsilon_{A,B}$ is invertible and its inverse is an elementary operator. 
Then $\Upsilon_{A,B}$ can be written as a sum of two invertible multiplication operators.
\end{lemma}

\begin{proof} 
By Theorem \ref{standard}, either $A$ or $B$ is algebraic. Hence $\sigma(A)$ or $\sigma(B)$ is finite. 
We can now deduce easily the desired decomposition.
\end{proof}

\begin{proposition}\label{inv} 
Let $\Delta$ be an invertible elementary operator of length $2$. Suppose that $\Delta^{-1}$ is an elementary operator of length $n$. 
Then one of the following cases holds.

{\rm (1)} $\Delta= \Mu_{A,B}+ \Mu_{C,D}$, where $\Mu_{A,B}, \Mu_{C,D}$ are invertible multiplication operators,
either $A^{-1}C$ or $B^{-1}D$ is algebraic and $\min\{\deg (A^{-1}C), \deg(B^{-1}D)\}=n$.

{\rm (2)} $\Dim(\sR(\Delta) \sR(\Delta^{-1})) \leq n$, every element of $\sR(\Delta)$ has a right zero divisor in $\sR(\Delta^{-1})$ and every element of 
$\sL(\Delta)$ has a right inverse in $\sL(\Delta^{-1})$.

{\rm (3)} $\Dim(\sL(\Delta^{-1})\sL(\Delta)) \leq n$, every element of $\sL(\Delta)$ has a left zero divisor in $ \sL(\Delta^{-1})$ and every element of 
$\sR(\Delta)$ has a left inverse in $\sL(\Delta^{-1})$.
\end{proposition}

\begin{proof} By Lemma \ref{dim},
$\Dim(\sR(\Delta) \sR( \Delta^{-1}))+ \Dim(\sL(\Delta^{-1}) \sL(\Delta)) \leq 2n+1$. This entails that either 
$\Dim(\sR(\Delta) \sR( \Delta^{-1}))\leq n$ or $\Dim(\sL(\Delta^{-1}) \sL(\Delta)) \leq n$. Suppose first that the former holds.
For $B \in \sR(\Delta)$, denote by $\widehat{B }: \sR(\Delta^{-1}) \rightarrow \sR(\Delta)\sR(\Delta^{-1})$ the restriction of $\rL_B$ to 
$\sR(\Delta^{-1})$. Let $\pS= \{\widehat{B} : B \in \sR(\Delta)\}$. Choose $B_1 \in \pS$ such that $\rk(B_1)= \rk(\pS)$. 
We distinguish two cases.

Case 1. $\rk(\pS) =n$.
Then $B_1 \sR(\Delta^{-1})= \sR(\Delta)\sR(\Delta^{-1})$ and $\widehat{B}_1$ is invertible. Choose $B_2 \in \sR(\Delta)$ such that 
$ \widehat{B}_2 $ is not invertible. Choose a Jordan basis $\{F_1, \ldots, F_n\}$ for the map $\widehat{B}_1^{-1}\widehat{B}_2$ such that $B_2 F_1=0$. 
Write $\Delta^{-1}= \sum_{i=1}^n \Mu_{E_i, F_i}$ and $\Delta= \sum_{i=1}^2 \Mu_{A_i, B_i}$ for suitable $A_1, A_2, E_1, \ldots, E_n \in \pL(\eX)$.  
Since $\Delta \Delta^{-1}=\Iota $ and $B_2 F_1=0$ one has
$ \sum_{j=1}^n A_1 E_j \otimes F_j B_1 F_1= I \otimes F_1$.
In particular, $I \in A_1 \sL(\Delta^{-1})$. Let $\{F_1, \ldots, F_{i_1}\}$ be the first block of the Jordan basis $\{F_1, \ldots, F_{n}\}$. 
Set $\eZ_1= \spa \{F_1, \ldots, F_{i_1}\}$ and $ \eZ_2=   \spa \{F_{i_1+1}, \ldots, F_n\}$. Then $B_2 \eZ_2 \subseteq B_1 \eZ_2$ and 
$B_2 F_{k+1}= B_1 F_k$ for $1 \leq k \leq i_1-1$. Choose $F \in \eZ_1$ and $G \in \eZ_2$ such that $B_1 (F+ G) =I$. Write 
$F= \sum_{k=1}^{i_1} \alpha_k F_k$ and let $l \in \{1, \ldots, i_1\}$ be the greatest integer satisfying $\alpha_l \neq 0$. 
It follows from $\Delta^{-1} \Delta =\Iota $ and Lemma \ref{inj} that
\begin{equation*}
\sum_{k=1}^{i_1-1} (E_{k+1}A_2+ E_k A_1) \otimes B_1 F_k + E_{i_1}A_1\otimes B_1 F_{i_1}+ \sum_{\substack{i_1+1 \leq k \leq n \\  1\leq j\leq 2}}E_k A_j \otimes B_j F_k =I .
\end{equation*}
If $l \neq i_1$, then $E_{i_1} A_1=0$, which is not possible as $I \in  A_1 \sL(\Delta^{-1})$. Hence $l=i_1$ and $E_{i_1} A_1= \alpha_lI$. 
This entails that $A_1$ is invertible. If $B_1$ is invertible, then we can apply Lemma \ref{suminv} and Theorem \ref{standard} to the operator 
$\Mu_{A, B}^{-1}\Delta$ and get the conclusion (1). Next suppose that $B_1$ is not invertible. Since $I \in B_1 \sR(\Delta^{-1})$, the restriction of the map 
$\rR_{B_1}$ to $\sR(\Delta^{-1})$ has to be injective. Hence $\Dim(\sR(\Delta^{-1}) \sR(\Delta)) \geq n$. Since $I \not\in \rR_{B_1} (\sR(\Delta^{-1}))$ one has 
$\Dim(\sR(\Delta^{-1}) \sR(\Delta)) \geq n+1$.  By Lemma \ref{dim}, $\Dim(\sL(\Delta) \sL(\Delta^{-1})) \leq n$. Since $A_1$ is invertible 
$\Dim(\sL(\Delta) \sL(\Delta^{-1}))= n$. The above argument applied to $A_1$ and $\sL(\Delta)\sL(\Delta^{-1})$ implies that $\Delta= \Mu_1+ \Mu_2$, 
where $\Mu_1$ and $\Mu_2$ are multiplication operators and $\Mu_1$ is invertible. The desired conclusion follows, once again, by Lemma \ref{suminv} 
and Theorem \ref{standard}.

Case 2. $\rk(\pS) \leq n-1$.
Let $A_1 \in  \sL(\Delta)$. Write $\Delta= \sum \Mu_{A_i, B_i}$ for suitable $A_2, B_1, B_2 \in \pL(\eX)$. Choose $F_1 \in \sR(\Delta^{-1})$ such that 
$B_2F_1=0$. Then $ \sum_{j=1}^n A_1 E_j \otimes F_j B_1 F_1= I \otimes F_1$.
Hence $A_1$ is right invertible.

Now suppose that $\Dim(\sL(\Delta^{-1}) \sL(\Delta) \leq n$. A reasoning similar to that one just presented gives that either (1) or (3) holds.
\end{proof}

\begin{corollary} 
Let $\Delta$ be an  elementary operator of length $2$ on $\bM_n$. Then one of the following conditions holds.

{\rm (1)} There exists a multiplication operator $\Mu$ such that $\Mu \Delta=0$.

{\rm (2)} $\Delta$ is a sum of two invertible multiplication operators and is invertible.
\end{corollary}

\begin{proof} Suppose first that $\Delta$ is invertible. If every nonzero element of $\sL(\Delta)$ is right invertible, then $\sL(\Delta)$ can be seen as 
a two dimensional subspace of constant rank $n$, which is impossible. Analogously, we see that there exists a nonzero element of $\sR(\Delta)$ which is not 
left invertible. It follows, by Proposition \ref{inv}, that $\Delta$ is a sum of two invertible multiplication operators. 
Assume now that $\Delta$ is not invertible and $\Mu \Delta \neq 0$, for every multiplication operator $\Mu$. Then, by Theorem \ref{M1}, either every   
element of  $\sL(\Delta)\setminus \{ 0\}$ is invertible or every element of $\sR(\Delta)\setminus \{ 0\}$ is invertible, which is impossible, as shown above. 
\end{proof}

\begin{theorem}
Let $\Delta  $ be an invertible elementary operator of length $2$. Suppose that $\Delta^{-1}$ is an elementary operator of length $2$, as well. Then  
there exist multiplication operators $\Mu_1, \Mu_2, \Gamma_1$, and $\Gamma_2$ such that  $\Delta= \Mu_1+ \Mu_2$, $\Delta^{-1}= \Gamma_1+\Gamma_2 $, 
and one of the following assertions holds.

{\rm (1)} $\Mu_1$ and $ \Mu_2$ are invertible.

{\rm (2)} $\Gamma_i \Mu_j= \Mu_j \Gamma_i=0$, for $i \neq j$.
\end{theorem}

\begin{proof}
Once again, for each $B \in \sR(\Delta)$, let $\widehat{B}: \sR(\Delta^{-1}) \rightarrow \sR(\Delta) \sR(\Delta^{-1})$ denote the restriction of $\rL_B$ to 
$\sR(\Delta)^{-1}$. In view of Proposition \ref{inv} it is enough to consider the case  when every element of $\sR(\Delta)$ has a right zero divisor in 
$\sR(\Delta)^{-1}$. Write $\Delta=\spa \{B_1, B_2\}$ and choose $F_1 \in \sR(\Delta)^{-1}\setminus\{ 0\}$ such that $B_2 F_1=0$. Then $B_1 F_1 \neq 0$. 
Choose $F_2 \in \sR(\Delta)^{-1}\setminus\{ 0\}$ such that $B_1 F_2 =0$. If $\{B_1 F_1, B_2 F_2\}$ was a linearly independent set, then 
$\{(B_1+B_2) F_1, (B_1+B_2) F_2\}$ would be linearly independent, which contradicts our assumption. On the other hand, $I \in \sR(\Delta) \sR(\Delta)^{-1}$. 
Thus, with no loss of generality, we may suppose that $B_1 F_1= B_2 F_2=I$. Write $\Delta= \sum \Mu_{A_i, B_i}$  and $\Delta^{-1}= \sum_{i=1}^2 \Mu_{E_i, F_i}$.  
Then it follows from our assumptions and Lemma \ref{inj} that  $(E_1 A_1+ E_2 A_2 )\otimes B_1F_1=I$. This implies that $ E_1A_1+E_2A_2=I$. 
Now we infer from the fact that $\Delta \Delta^{-1}= I$ that
$\sum A_1 E_i \otimes F_iB_1+ \sum A_2 E_i \otimes F_i B_2= I$.
This entails that $\sum A_1 E_i  \otimes F_i=I \otimes F_1$ and $\sum A_2E_i \otimes F_i= I \otimes F_2$.
Consequently, $A_iE_i=I$ and $A_i E_j=0$, for $i \neq j$. Therefore, we have $F_1B_1+F_2B_2=I$. Set $\Gamma_i= \Mu_{E_i, F_i}$ and $\Mu_i= \Mu_{A_i, B_i}$. 
Then  $\Gamma_i \Mu_j= \Mu_j \Gamma_i=0 $, as desired.
\end{proof}

The following example illustrates  the second case in the above theorem.

\begin{example}
Let $\eH$ be the separable Hilbert space and let $\{e_i\}_{i=1}^{\infty}$ be an orthonormal basis. 
Let the operators $B_1, B_2$ be defined by 
\begin{equation*}
B_1 e_{2i+1}=0,\qquad B_2 e_{2i+1}=e_{i+1} \qquad (i \geq 0)
\end{equation*} 
and 
\begin{equation*}
B_1  e_{2i}=e_i,\qquad B_2 e_{2i}= 0 \qquad (i \geq 1).
\end{equation*} 
Set $ \Delta= \Mu_{B_1, B_2}+ \Mu_{B_2, B_1}$. Then $\Delta$ is invertible and its inverse is $\Mu_{F_1, F_2}+\Mu_{F_2, F_1}$, where
\begin{equation*} 
F_1 e_i= e_{2i}\quad \text{and}\quad  F_2 e_i= e_{2i-1}\qquad (i \geq 1).
\end{equation*}  
Observe that $\Delta$ cannot be a sum of two invertible multiplication operators as the pencil $(B_1, B_2)$ is not regular. 
Moreover, a straightforward calculation shows that every elementary operator $\Psi$ satisfying $\sL(\Psi)= \sL(\Delta)$ and 
$\sR(\Psi)= \sR(\Delta)$ has to be invertible and its inverse is a length $2$ elementary operator.
\end{example}


\end{document}